\documentclass[letterpaper,11pt]{amsart}

\usepackage[margin=1.2in]{geometry}
\usepackage{amsmath,amsthm,amssymb}
\usepackage{xspace,xcolor}
\usepackage[breaklinks,colorlinks,citecolor=teal,linkcolor=teal,urlcolor=teal,pagebackref,hyperindex]{hyperref}
\usepackage[alphabetic]{amsrefs}
\usepackage[all]{xy}

\theoremstyle{plain}
\newtheorem{thm}{Theorem}
\newtheorem{lem}[thm]{Lemma}
\newtheorem{prop}[thm]{Proposition}

\theoremstyle{definition}
\newtheorem{defi}[thm]{Definition}

\newtheorem{question}[thm]{Question}

\theoremstyle{remark}
\newtheorem{rmk}[thm]{Remark}

\numberwithin{equation}{section}

\def\Z{{\mathbb Z}}
\def\Q{{\mathbb Q}}

\def\J{\mathcal{J}}
\def\O{\mathcal{O}}

\def\p{\pi}

\def\t{\tau}

\def\D{\Delta}
\def\G{\Gamma}

\def\.{\cdot}
\def\~{\widetilde}

\def\lru{\lceil}
\def\rru{\rceil}

\def\({\left(}
\def\){\right)}

\newcommand{\ru}[1]{\lru{#1}\rru}

\renewcommand{\and}{ \ \ \text{ and } \ \ }

\DeclareMathOperator{\Tr}  {Tr}

\DeclareMathOperator{\Spec} {Spec}

\DeclareMathOperator{\Ex} {Ex}

\DeclareMathOperator{\Supp} {Supp}

\DeclareMathOperator{\Hom} {Hom}

\begin{document}

\title[Comparing multiplier ideals to test ideals]
{Comparing multiplier ideals to test ideals on numerically $\Q$-Gorenstein varieties}

\author{Tommaso de Fernex}
\address{Department of Mathematics, University of Utah,
155 South 1400 East, Salt Lake City, UT 48112-0090, USA}
\email{{\tt defernex@math.utah.edu}}

\author{Roi Docampo}
\address{Instituto de Matem\'{a}tica e Estat\'{\i}stica, Universidade Federal Fluminense,
Rua M\'{a}rio Santos Braga, s/n, 24020-140 Niter\'{o}i, RJ, Brasil}
\email{{\tt roi@mat.uff.br}}

\author{Shunsuke Takagi}
\address{Graduate School of Mathematical Sciences, 
the University of Tokyo, 3-8-1 Komaba, Meguro-ku, Tokyo 153-8914, Japan.}
\email{{\tt stakagi@ms.u-tokyo.ac.jp}}

\author{Kevin Tucker}
\address{Department of Mathematics, Statistics, and Computer Science, 
University of Illinois at Chicago,
851 S. Morgan Street, Chicago, IL 60607-7045, USA}
\email{{\tt kftucker@uic.edu}}


\thanks{The first author 
was partially supported by NSF CAREER grant DMS-0847059, 
NSF FRG Grant DMS-1265285, and a Simons Fellowship. 
The second author was partially supported by
CNPq/Ci\^encia Sem Fronteiras, under the program Atra\c{c}\^{a}o de
Jovens Talentos, Processo 370329/2013--9.
The third author was partially supported by Grant-in-Aid for Young
Scientists (B) 23740024 from JSPS.  The fourth author was partially
supported by NSF grant DMS-1419448.}

\thanks{Compiled on \today. Filename {\tt \jobname}}

\begin{abstract}
We show that the reduction to positive characteristic of the multiplier
ideal in the sense of de~Fernex and Hacon agrees with the test ideal for
infinitely many primes, assuming that the variety is numerically
$\Q$-Gorenstein. It follows, in particular, that this reduction property holds 
in dimension 2 for all normal surfaces. 
\end{abstract}

\maketitle


\section*{Introduction}

It has been understood for a while now that there is a connection between
the theory of multiplier ideals in characteristic zero and the theory of
test ideals in positive characteristic (cf.\ \cite{Smi00,Har01,Tak04,MS11,BST}). 
But one of the puzzling features of
this connection is its asymmetry: test ideals are defined for arbitrary
varieties, but multiplier ideals are only defined after imposing some
conditions on the singularities (namely one requires the variety to be
(log) $\Q$-Gorenstein).
It is therefore natural to wonder if there are extensions of the theory of
multiplier ideals beyond the $\Q$-Gorenstein case, and what could be their
relation to test ideals. One such possible generalization has been proposed
in~\cite{dFH09}, where only normality is assumed to define the multiplier
ideal. This theory is still largely unexplored, and few tools are available
for its study. 

In~\cite{BdFF12,BdFFU}, the authors introduce and study a new
class of singularities, called \emph{numerically
$\Q$-Gorenstein} singularities, which encompasses the class of $\Q$-Gorenstein singularities.
It turns out that the theory of multiplier ideals on normal varieties, as introduced in \cite{dFH09}, 
becomes particularly simple on numerically $\Q$-Gorenstein varieties. 
One key feature is that while on arbitrary normal varieties the definition 
of multiplier ideals requires an asymptotic construction which involves infinitely many 
resolutions, multiplier ideals
on numerically $\Q$-Gorenstein varieties can always be computed from
one resolution of singularities, just like in the $\Q$-Gorenstein case. 

The purpose of this note is the following
theorem, where multiplier ideals in the sense of~\cite{dFH09} are related
to test ideals assuming the variety is numerically $\Q$-Gorenstein.

\begin{thm}
\label{t:main}
Let $X$ be a numerically $\Q$-Gorenstein variety over an uncountable algebraically
closed field $k$ of characteristic~0, and let $Z$ be an effective
$\Q$-Cartier $\Q$-divisor on $X$. Given a model of $(X, Z)$ over a finitely
generated $\Z$-subalgebra $A$ of $k$, 
there exists a dense open subset $S \subseteq \Spec A$ such that for all
closed points $\mu \in S$,  
\[
\J(X,Z)_{\mu} = \t(X_{\mu}, Z_{\mu}).
\]
Here $\J(X,Z)$ denotes the multiplier ideal of the pair $(X,Z)$ in the
sense of~\cite{dFH09}, and $\t(X_{\mu}, Z_{\mu})$ is the $($big$)$ test ideal
of the pair $(X_{\mu}, Z_{\mu})$.
\end{thm}

Since every normal surface is
numerically $\Q$-Gorenstein (cf.\ \cite[Example~2.29]{BdFF12}), 
the theorem applies, in particular, 
to all pairs $(X,Z)$ where $X$ is any normal surface and $Z$ is an effective
$\Q$-Cartier $\Q$-divisor on $X$.

Theorem~\ref{t:main} is an extension of the 
results on (log) $\Q$-Gorenstein varieties \cite{Smi00,Har01,Tak04}.
It remains open, at the moment, whether the same property holds on all normal varieties
(this question is considered for instance in \cite[Remark~6.2]{Sch11}).
We hope that our theorem, apart from bringing evidence to this,
will also show that numerically $\Q$-Gorenstein varieties
provide a natural framework for the theory of singularities of pairs. 

The paper is organized as follows. In section~\ref{s:multiplier-ideals}
we review the theories of multiplier ideals in the sense of~\cite{dFH09}
and of numerically $\Q$-Gorenstein singularities. In
section~\ref{s:test-ideals} we discuss the theory of test ideals and the
previously known relations to multiplier ideals. In
section~\ref{s:proof}, we give the proof Theorem~\ref{t:main}.
In the last section, we briefly discuss the numerically
$\Q$-Gorenstein condition in positive characteristics and ask some natural
open questions.

\subsection*{Acknowledgments} 
This paper is the result of a project started during the AIM workshop
``Relating test ideals and multiplier ideals,'' Palo Alto, August 8--12, 2011.

\section{Multiplier ideals}
\label{s:multiplier-ideals}

Consider a pair $(X, Z)$, where $X$ is a normal variety, defined over an
algebraically closed field $k$ of characteristic~0, and $Z$ is an
effective $\Q$-Cartier $\Q$-divisor on $X$. The normality of $X$ guarantees
the existence of canonical divisors $K_X$ on $X$.

The classical way of associating a multiplier ideal to $(X,Z)$ assumes that
$K_X$ is $\Q$-Cartier (the $\Q$-Gorenstein condition) or, more generally,
that we fix an effective $\Q$-divisor $\D$ such that $K_X + \D$ is
$\Q$-Cartier. Then we consider a log resolution $\pi\colon Y \to X$ of
$((X,\D),Z)$ and pick canonical divisors $K_X$ on $X$ and $K_Y$ on $Y$ in a
compatible way: $\pi_*K_Y = K_X$.

\begin{defi}
The sheaf of ideals
\[
\J((X,\D),Z) := \p_*\O_Y(\ru{K_Y - \p^*(K_X + \D + Z)})
\]
is called the \emph{multiplier ideal} of $((X,\D),Z)$. When $X$ is
$\Q$-Gorenstein we write $\J(X,Z)$ instead of $\J((X,0),Z)$.
\end{defi}

In~\cite{dFH09} a multiplier ideal is defined without using an auxiliary
boundary $\D$. For this, consider log resolutions $\pi_m \colon Y_m \to X$
of $(X,Z)$ with the following extra conditions: the inverse image
of the fractional ideal $\O_X(-mK_X)$ is of the form
\[ 
\O_X(-mK_X) \. \O_{Y_m}
=
\O_Y(-f^\sharp(mK_X)),
\]
for some Cartier divisor $f^\sharp(mK_X)$ on $Y_m$ with simple normal
crossings with $\Ex(\pi_m) \cup \Supp(\pi_m^*Z)$. After picking canonical
divisors $K_{Y_m}$ and $K_X$ with $\pi_*K_{Y_m} = K_X$, we define
\[
K_{m, Y_m/X} := K_{Y_m} - \tfrac 1m f^\sharp(mK_X).
\]

\begin{defi}
The sheaf of ideals
\[
\J_m(X,Z) := 
{\p_m}_*\O_{Y_m}(\ru{K_{m, Y_m/X} - \p^*Z})
\]
is called the \emph{$m$-limiting multiplier ideal} of $(X,Z)$. The
\emph{multiplier ideal} $\J(X,Z)$ of the pair $(X,Z)$ is the maximal
element of the set of $m$-limiting multiplier ideals:
\[
\J(X,Z) := \J_{m_0}(X,Z) = \sum_{m=1}^\infty \J_m(X,Z),
\]
where $m_0$ is sufficiently divisible (the condition on divisibility depends both on $X$ and $Z$). 
\end{defi}

If $K_X$ is $\Q$-Cartier, then the sequence of $m$-limiting multiplier
ideals stabilizes to the classical $\J(X,Z)$ as soon as $m$ is divisible by
the Cartier index of $K_X$, and in particular both definitions of $\J(X,Z)$
agree in this case. The following result relates the two notions in
general. 

\begin{thm}[\protect{\cite[Theorem~1.1]{dFH09}}]
\label{t:dFH}
The multiplier ideal $\J(X,Z)$ is the unique maximal element of the set of multiplier ideals
$\J((X,\D),Z)$ where $\D$ varies among all effective $\Q$-divisors such that
$K_X + \D$ is $\Q$-Cartier. In particular, there is one such $\D_0$ such that 
\[
\J(X,Z) = \J((X,\D_0),Z)
= \sum_\D
\J((X,\D),Z).
\]
\end{thm}

In general it is not clear how to detect when the sequence of $m$-limiting
multiplier ideals stabilizes. A priori, one might need to take into account
infinitely many resolutions to determine that we have computed $\J(X,Z)$.
The condition of \emph{numerically $\Q$-Gorenstein} singularities,
introduced and studied in~\cite{BdFF12,BdFFU}, is designed to provide a
work-around for this issue.\footnote{The definitions of \emph{numerically $\Q$-Cartier}
and \emph{numerically $\Q$-Gorenstein} given here follow \cite{BdFFU};
they are equivalent to the definitions of \emph{numerically Cartier}
and \emph{numerically Gorenstein} given in \cite{BdFF12}.}

\begin{defi}
\label{d:num-Gor}
A (Weil) divisor $D$ on a normal variety $X$ is said to be
\emph{numerically $\Q$-Cartier} if for some (equivalently any) resolution
of singularities $\pi \colon Y \to X$, there is a ($\Q$-Cartier)
$\Q$-divisor $D'$ on $Y$ that is $\pi$-numerically trivial (that is, it has zero
intersection with all $\pi$-exceptional curves on $Y$) and satisfies $D = \pi_* D'$.
When $D$ is numerically $\Q$-Cartier, the $\Q$-divisor $D'$ as above is
uniquely determined for each resolution $Y$. We denote it $\pi^*_{\rm
num}D$ and call it the \emph{numerical pull-back} of $D$ to $Y$.
\end{defi}

\begin{defi}
When $K_X$ is numerically $\Q$-Cartier, then we say that $X$ is
\emph{numerically $\Q$-Gorenstein}. In this case, given a resolution
$\pi\colon Y \to X$, and canonical divisors $K_Y$ and $K_X$ with $\pi_*K_Y
= K_X$, we define
\[
K_{Y/X}^{\rm num} := K_Y - \p^*_{\rm num}K_X.
\]
\end{defi}

In the case of surfaces, the notion of numerical pull-back was introduced
by Mumford. In fact, it follows from the Negativity Lemma that all divisors
on normal surfaces are numerically $\Q$-Cartier, and in particular all
normal surfaces are numerically $\Q$-Gorenstein.

Another useful example to keep in mind in order 
to compare this notion with the more common notion of $\Q$-Gorenstein singularities
is the case of cones. Suppose $X$ is the cone over a projective manifold $V$ polarized by a
projectively normal divisor $L$ (so that $X$ is normal). Then $X$ is
$\Q$-Gorenstein (resp., numerically $\Q$-Gorenstein) if and only if
$L$ is $\Q$-linearly proportional (resp., numerically proportional) to $K_V$
(cf.\ \cite[Lemma~2.32]{BdFF12}). 

The main reason the notion of numerical $\Q$-Gorenstein singularities is
useful in the context of multiplier ideals is the following result.

\begin{thm}[\protect{\cite[Theorem~1.3]{BdFFU}}]
\label{t:BdFFU}
Assume that $X$ is numerically $\Q$-Gorenstein. For every log resolution
$\p \colon Y \to X$ of $(X,Z)$, we have
\[
\J(X,Z) = \p_*\O_Y(\ru{K_{Y/X}^{\rm num} - \p^*Z}).
\]
\end{thm}

\section{Test ideals}
\label{s:test-ideals}

A field $k$ of characteristic $p>0$ is said to be \textit{$F$-finite} if $[k:k^p]<\infty$. 
Let $V$ be a normal variety defined over an $F$-finite field of
characteristic $p > 0$, and let $W$ be an effective $\Q$-divisor on $V$. 
Let $F^e : V \to V$ denote the $e^{\rm th}$ iteration of the
absolute Frobenius morphism.

\begin{defi}
We define the \textit{test ideal}\footnote{It has been traditionally called
in the literature the \emph{big} (or \emph{non-finitistic}) \emph{test
ideal}. } $\tau(V, W)$ of the pair $(V, W)$ to be the smallest nonzero
ideal sheaf $J \subseteq \O_V$ which locally satisfies the following condition: 
for every $e \ge 0$ and every map $\phi \in
\mathrm{Hom}_{\O_V}(F^e_* \O_V(\lceil (p^e - 1)W \rceil), \O_V)$,
we have
    \[
    \phi\big(F^e_*J \big) \subseteq J.
    \]
\end{defi}

\begin{thm}[\protect{\cite[Corollary~5.2]{Sch11}}]
\label{t:Sch}
Suppose that $W$ is an effective $\Q$-Cartier $\Q$-divisor on $V$. 
Then we have 
\[
\t(V,W) = \sum_\G \t(V,\G+W)
\]
where the sum runs over all effective $\Q$-divisors $\G$ on $V$ such that
$K_V + \G$ is $\Q$-Cartier. 
\end{thm}

We now recall how test ideals behave under reduction to characteristic
zero.  We refer to the sources \cite[Section 2.2]{MS11} and \cite{HH99} for
a detailed description of the process of reduction to characteristic zero.

Let $X$ be a normal scheme of finite type over a field $k$ of
characteristic zero and $Z \subsetneq X$ be a closed subscheme. Choosing a
suitable finitely generated $\Z$-subalgebra $A$ of $k$, we can construct a
scheme $X_A$ of finite type over $A$ and closed subschemes $Z_{A}
\subsetneq X_A$ such that $(Z_{A} \hookrightarrow X_A) \otimes_A k \cong (Z
\hookrightarrow X)$. We can enlarge $A$ by localizing at a single nonzero
element and replacing $X_A$ and $Z_{A}$ with the corresponding open
subschemes. Thus, applying generic freeness \cite[(2.1.4)]{HH99}, we may
assume that $X_A$ and $Z_{A}$ are flat over $\Spec A$. We refer to $(X_A,
Z_A)$ as a \textit{model} of $(X, Z)$ over $A$.
If $Z$ is a prime divisor on $X$, then possibly enlarging $A$, we may
assume that $Z_A$ is a prime divisor on $X_A$. When $D=\sum_i d_i D_i$ is a
$\Q$-divisor on $X$, let $D_A:=\sum_i d_i D_{i, A}$.

Given a closed point $\mu \in \Spec A$, we denote by $X_{\mu}$
(resp.~$Z_{\mu}$, $D_{i, \mu}$) the fiber of $X_A$ (resp.~$Z_{A}$,
$D_{i,A}$) over $\mu$ and denote $D_{\mu}:=\sum_i d_i D_{i, \mu}$. Then
$X_{\mu}$ is a scheme of finite type over the residue field $\kappa(\mu)$
of $\mu$, which is a finite field of characteristic~$p(\mu)$. After
enlarging $A$ if necessarily, $X_\mu$ is normal, the $D_{i,\mu}$ are prime
divisors, and $D_\mu$ is a $\Q$-divisor on $X_\mu$ for all closed points
$\mu \in \Spec A$.

\begin{question}
\label{q:main}
Let $X$ be a normal variety over an algebraically closed field $k$ of
characteristic~0 and $Z$ be an effective $\Q$-Cartier $\Q$-divisor on
$X$. Suppose we are given a model of $(X, Z)$ over a finitely generated
$\Z$-subalgebra $A$ of $k$. Then does there exist a dense open subset $S
\subseteq \Spec A$ such that 
\[
\J(X,Z)_{\mu}= \t(X_{\mu},Z_{\mu})
\]
for every closed point $\mu \in S$?
\end{question}

When $K_X$ is $\Q$-Cartier, this property was established by Smith and Hara
\cite{Smi00,Har01}. Their result was extended to all log pairs by Takagi, 
who proved the following result.

\begin{thm}[\protect{\cite[Theorem~3.2]{Tak04}}]
\label{t:Tak}
With the above notation, suppose that $\D$ is an effective $\Q$-divisor on $X$
such that $K_X + \D$ is $\Q$-Cartier. 
Given a model of $(X, \D, Z)$ over a finitely generated $\Z$-subalgebra $A$ of $k$, 
there exists a dense open subset $S$ of $\Spec A$ such that for all closed points $\mu \in S$, 
\[
\J((X,\D),Z)_{\mu} = \t(X_{\mu},\D_{\mu}+Z_{\mu}).
\]
\end{thm}

Question~\ref{q:main} remains in open in general. 
Theorem~\ref{t:main} settles the question when $K_X$ is numerically $\Q$-Cartier

\begin{rmk}
In the setting of Theorem~\ref{t:Sch}, it is an open question whether
$\t(V,W)$ is actually equal to $\t(V,\G+W)$ for some effective
$\Q$-divisor $\G$ such that $K_V + \G$ is $\Q$-Cartier. Furthermore,
suppose $V=X_{\mu}$ is a general closed fiber of a deformation $X_A \to \Spec
A$, where $X$ is a normal variety defined over a field of
characteristic~0 and $A$ is a finitely generated $\Z$-subalgebra of $k$.
Similarly, suppose $W=Z_{\mu}$ is obtained from an effective $\Q$-divisor
$Z$ on $X$. If we assume that there exists an effective $\Q$-divisor
$\Gamma$ on $V$ such that $\t(X_{\mu},Z_{\mu}) =
\t(X_{\mu},\Gamma+Z_{\mu})$, it is also an open question whether such
$\G$ could be reduced from characteristic~0 (cf.\
\cite[Question~6.3]{Sch11}).
\end{rmk}

\section{Proof of the theorem}
\label{s:proof}

As in Theorem~\ref{t:main}, let $X$ be a numerically $\Q$-Gorenstein
variety over an uncountable algebraically closed field $k$ of characteristic~0, and
let $Z$ be an effective $\Q$-Cartier $\Q$-divisor on $X$. Suppose we are
given a model of $(X, Z)$ over a finitely generated $\Z$-subalgebra $A$ of
$k$. Let $\p \colon Y \to X$ be a fixed projective log resolution of
$(X,Z)$. After possibly enlarging $A$, we may assume that we have a model
of $\pi$ over $A$, and that $\pi_{\mu}$ is a projective log resolution of
$(X_{\mu}, Z_{\mu})$ for all closed points $\mu \in \Spec A$. 

It follows from Theorem~\ref{t:dFH} that  
\[
\J(X,Z) = \J((X,\D),Z)
\]
for some effective $\Q$-divisor $\D$ on $X$ such that $K_X + \D$ is
$\Q$-Cartier. Enlarging $A$ if necessary, we may assume that a model of
$\Delta$ over $A$ is given, $K_{X_{\mu}}+\Delta_{\mu}$ is $\Q$-Cartier and 
\[
\J(X,Z)_{\mu} = \J((X,\D),Z)_{\mu} 
\]
for  all closed points $\mu \in \Spec A$. By Theorem~\ref{t:Tak}, after
possibly enlarging $A$, we may assume that
\[
\J((X,\D),Z)_{\mu} = \t(X_{\mu},\D_{\mu}+Z_{\mu})
\]
for all closed points $\mu \in \Spec A$.
Then, since by Theorem~\ref{t:Sch} we have
\[
\t(X_{\mu},Z_{\mu}) = \sum_\G \t(X_{\mu},\G+Z_{\mu}),
\]
where the sum runs over all effective $\Q$-divisors $\G$ on $X_{\mu}$ such
that $K_{X_{\mu}} + \G$ is $\Q$-Cartier, we obtain 
\[
\J(X,Z)_{\mu} \subseteq \t(X_{\mu}, Z_{\mu})
\]
for all closed points $\mu \in \Spec A$. 

To finish the proof, we need to show that the reverse inclusion 
\begin{equation}
\label{eq:normallyeasyinclusion}
\t(X_{\mu},Z_{\mu}) \subseteq \J(X,Z)_{\mu} 
\end{equation}
also holds for almost all $\mu \in \Spec A$. It suffices to show that there
exists a dense open subset $S \subseteq \Spec A$ such that for every closed
point $\mu \in S$ and every effective $\Q$-divisor $\G$ on $X_{\mu}$ such
that $K_{X_{\mu}}+\G$ is $\Q$-Cartier,  one has 
\[
\t(X_{\mu},\G+Z_{\mu}) \subseteq \J(X,Z)_{\mu}. 
\]
On the one hand, by \cite[Theorem 2.13]{Tak04} (see also \cite{BST}) there
is always an inclusion
\[
\t(X_\mu,\G+Z_\mu) \subseteq 
{\p_{\mu}}_*\O_{Y_{\mu}}
(\ru{K_{Y_{\mu}} - \p_\mu^*(K_{X_{\mu}}+\G + Z_{\mu})}).
\]
Since by Theorem~\ref{t:BdFFU} we have
\[
\J(X,Z) = \p_*\O_Y(\ru{K_{Y/X}^{\rm num} - \p^*Z}),
\]
after possibly enlarging $A$, we may assume that 
\[
\J(X,Z)_{\mu} 
= 
{\p_{\mu}}_*\O_{Y_{\mu}}
(\ru{(K_{Y/X}^{\rm num})_{\mu} - \p_{\mu}^*Z_{\mu}})
\]
for all closed points $\mu \in \Spec A$. Therefore, the proof of the
theorem is complete once we establish the following property. 

\begin{prop}
\label{p:ineq}
There exists a dense open subset $S \subset \Spec A$ such that for every
closed point $\mu \in S$,  
\[
K_{Y_{\mu}} - \p_{\mu}^*(K_{X_{\mu}} + \G) 
\le 
(K_{Y/X}^{\rm num})_\mu.
\]
\end{prop}

The proof of the proposition uses the Negativity Lemma (cf.\
\cite[Lemma~3.39]{KM98}) in the following more precise formulation. We say
that an irreducible projective curve $C$ in a variety $X$ is \emph{movable
in codimension~1} if it belongs to an irreducible family of curves whose
locus in $X$ is dense in a codimension~1 subset of $X$. 

\begin{lem}
\label{l:NL}
Let $\p \colon Y \to X$ be a projective resolution of a normal variety over
an uncountable algebraically closed field of characteristic~0, and assume that
the exceptional locus of $\p$ has pure codimension~1. Then there are finitely many irreducible
curves $C_1,\dots,C_s$ in $Y$ such that 
\begin{enumerate}
\item
each $C_j$ is $\p$-exceptional and movable in codimension~1, and
\item
for every divisor $D$ on $Y$, if $D\.C_j \ge 0$ for all $j$ and 
$\p_*D \le 0$, then $D \le 0$. 
\end{enumerate}
\end{lem}

\begin{proof}[Proof of Proposition~\ref{p:ineq}]
Let $E_1, \dots, E_t$ be the prime exceptional divisors of $\p \colon Y \to
X$, and let $C_1, \dots, C_s$ be the irreducible curves on $Y$ satisfying
the properties listed in Lemma~\ref{l:NL}. 

By possibly enlarging $A$, we can assume that $E_{1,\mu}, \dots, E_{t,\mu}$
are the prime exceptional divisors of $\p_\mu \colon Y_\mu \to X_\mu$, and
that $C_{1,\mu}, \dots, C_{s,\mu}$ are $\p_\mu$-exceptional and movable in
codimension~1. Furthermore, we assume that the reduction is such that
$E_{i,\mu} \. C_{j,\mu} = E_i \. C_j$ and $K_{Y_\mu}\.C_{j,\mu} =  K_Y \.
C_j$ for all $i,j$ and all $\mu$. 

Fix $\mu \in \Spec A$. We consider the $\Q$-divisor
\[
F_\mu := 
K_{Y_\mu} - \p_\mu^*(K_{X_\mu} + \G) - 
(K_{Y/X}^{\rm num})_\mu + (\p_\mu)_*^{-1}\G,
\]
where $\G$ is an arbitrary effective $\Q$-divisor such that $K_{Y_\mu} +
\G$ is $\Q$-Cartier. We need to show that $F_\mu$ is anti-effective. Since it is
exceptional, we can write
\[
F_\mu = \sum a_{i} E_{i,\mu}.
\]
Then $F_\mu$ lifts to characteristic~0, as it is the reduction of the
divisor 
\[
F = \sum a_i E_i.
\] 
So it suffices to check that $F \le 0$. To this end, by Lemma~\ref{l:NL} we
only need to check that $F\.C_j \ge 0$ for all $j$. 

Note that $\p_\mu^*(K_{X_\mu} + \G)\.C_{j,\mu} = 0$ since $\p_\mu^*(K_{X_\mu}
+ \G)$ is $\p_\mu$-numerically trivial. Moreover, since $C_{i,\mu}$ is
$\pi_\mu$-exceptional and movable in codimension~1, and
$(\pi_\mu)_*^{-1}\G$ does not contain any $\p_\mu$-exceptional component,
we have that $(\p_\mu)_*^{-1}\G \. C_{j,\mu} \ge 0$. Therefore:
\begin{align*}
F \. C_j 
&= F_\mu \. C_{j,\mu} \\
&\ge (K_{Y_\mu} - (K_{Y/X}^{\rm num})_\mu)\.C_{j,\mu} \\
&=(K_Y - K_{Y/X}^{\rm num})_\mu\.C_{j,\mu} \\
&=(K_Y - K_{Y/X}^{\rm num})\.C_j \\
&=(\pi^*_{\rm num}K_X) \. C_j \\
&= 0.
\end{align*}
The last equality follows from the fact that $\pi^*_{\rm num}K_X$ is, by
definition, $\pi$-numerically trivial, and that $C_j$ is $\pi$-exceptional.
\end{proof}

\begin{proof}[Proof of Lemma~\ref{l:NL}]
Our goal is to reduce to the usual Negativity Lemma for surfaces
(essentially the Hodge Index Theorem, see~\cite[Lemmas~3.40, 3.41]{KM98}).
The idea is that we can check for anti-effectivity of a divisor $D \subset
Y$ by considering its restrictions $D|_T$, where $T \subset Y$ is a
subsurface. But $D|_T$ will be anti-effective if intersects non-negatively
with the $\p$-exceptional curves of $T$. The problem is that in principle
one might need to use a different $T$ depending on the divisor $D$, and we
are lead to an infinite family of curves. The main observation is that,
given an integer $e$, it is possible to use the same $T$ for all divisors
$D$ whose irreducible components have image in $X$ of codimension $e$.

In more detail, we proceed as follows. First we assume that $X$ is
quasi-projective, which is allowed because the problem is local with
respect to $X$. Let $n \geq 3$ be the dimension of $X$. We fix very general
hyperplane sections $L_3, \dots, L_{n}$ in $X$, and $H_3, \dots, H_{n}$ in
$Y$. For every $e = 2, 3, \dots, n$ we consider the surface $T_e \subset Y$
given by
\[
T_e = 
(H_3 \cap \dots \cap H_e)
\cap 
\p^{-1}(L_{e+1} \cap \dots \cap L_n).
\]
Notice that $T_e$ is smooth. We let $S_e$ be the normalization of
$\p(T_e)$, and $\p_e\colon T_e \to S_e$ the map induced by $\p$. The
exceptional locus of $\p_e$ consists of finitely many curves, which we
denote $C_{e,1}, \dots, C_{e,m_e}$. We will show that the finite set of
curves
\[
\{\ C_{e,j} \ \mid \ e=2,3,\dots,n, \quad j=1,2,\dots,m_e \ \}
\]
has the properties required by the lemma.

Since $T_e$ is cut out by very general members in base-point free linear
systems, each of these curves $C_{e,j}$ is movable in codimension~1
inside $X$, and part (a) of the lemma holds. To prove part $(b)$, 
we need more information on how these curves $C_{e,j}$ intersect the irreducible
components of the exceptional locus of $\p$.

Let $E$ be an exceptional prime divisor, and let 
\[
B := \pi(E) \cap L_{e+1} \cap \dots \cap L_{n}.
\]
Note that
\[
T_e \cap E = H_3 \cap \dots \cap H_e \cap (\p|_E)^{-1}(B),
\]
where $\p|_E \colon E \to X$ is the restriction of $\p$ to $E$. 
Note also that $T_e \cap E$, if nonempty, has pure dimension 1. 
Denoting by $c$ the codimension of $\p(E)$ in $X$, we have the following cases:
\begin{enumerate}
\item[(i)]
If $c > e$, then $B$ is empty and $E$ is disjoint from $T_e$. In particular, $E\.C_{e,j} = 0$ for all $j$. 
\item[(ii)]
If $c = e$, then $B$ is zero-dimensional and $(\p|_E)^{-1}(B)$ is 
a union of general fibers of $\p|_E$. In particular, $T_e \cap E$ is nonempty and
each irreducible component is equal to $C_{e,j}$ for some $j$. 
\item[(iii)]
If $c < e$, then $B$ is an irreducible set of positive dimension, and 
$(\p|_E)^{-1}(B)$ is irreducible as well. Then $T_e \cap E$ is an irreducible
curve and is not $\p$-exceptional. In particular, $E$ does not contain $C_{e,j}$
for any $j$, and hence $E\.C_{e,j} \ge 0$ for all $j$.
\end{enumerate}

We now prove part (b) of the lemma. Suppose $D$ is a divisor on $Y$ such
that $\p_*D \le 0$ and $D\.C_{e,j} \ge 0$ for all $e,j$. Write $D = D_1 +
\dots + D_n$ where each $D_e$ is supported exactly on the irreducible
components of $D$ whose image in $X$ has codimension $e$. We claim that
$D_e \le 0$ for every $e$. We prove this by induction on $e$. For $e=1$,
this is the condition that $\p_*D \le 0$. 

Assume $e>1$. 
By~(ii), $T_e$ intersects properly every component of $D_e$, and
if $E$ is irreducible component of $D_e$ then $T_e \cap E$ is a general complete
intersection in $(\p|_E)^{-1}(B)$.
If $E'$ is a different irreducible component of $D_e$, 
and $B' := \p(E') \cap L_{e+1} \cap \dots \cap L_{n}$, 
then $(\p|_E)^{-1}(B)$ and $(\p|_{E'})^{-1}(B')$ do not have any components in common, 
since the hyperplane sections $L_{e+1},\dots,L_{n}$ are general. It follows that
the restrictions $E|_{T_e}$ and $E'|_{T_e}$ share no
common components. Therefore it is enough to show that $D_e|_{T_e} \leq 0$. 
Again by~(ii), we see that the irreducible
components of $D_e|_{T_e}$ are of the form $C_{e,j}$ for some $j$ (that is,
$D_e|_{T_e}$ is $\p_e$-exceptional). In particular, using the Negativity
Lemma for surfaces, it is enough to show that
$D_e|_{T_e}\.C_{e,j} \geq 0$ for all $j$. Since, by~(i), $D_{e'}\.C_{e,j} = 0$ whenever 
$e'>e$, we have that
\[
D \. C_{e,j} =
(D_1 + \dots + D_e) \. C_{e,j} =
(D_1 + \dots + D_{e-1}) \. C_{e,j} + D_e \. C_{e,j} .
\]
By induction we know that $D_1 \leq 0, \dots, D_{e-1} \leq 0$, and we know by~(iii) that
$C_{e,j}$ is not contained in the support of
$D_2 + \dots + D_{e-1}$. Moreover, $C_{e,j}$ is movable in codimension 1
and since it is $\p$-exceptional, the locus spanned by its deformations is $\p$-exceptional. 
This implies that the intersection of $C_{e,j}$ with each of the components of $D_1$ is
non-negative. Therefore we see that $(D_1 + \dots + D_{e-1}) \. C_{e,j}
\leq 0$. By hypothesis we know that $D \. C_{e,j} \geq 0$, so we must have
that $D_e|_{T_e} \. C_{e,j} = D_e \. C_{e,j} \geq 0$, as required.
\end{proof}

\section{Some open questions}
\label{sec:some-open-questions}

The notion of a numerically
$\Q$-Cartier divisor given above and studied in~\cite{BdFF12,BdFFU}
is, at first glance, particular to characteristic zero where
resolutions of singularities are known to exist.  However, in
arbitrary characteristic, one may just as well make use of the regular
alterations provided by \cite{dJ}.  Recall that a regular alteration
of a variety $X$ is a proper surjective morphism $\pi\colon Y \to X$
where $Y$ is regular and $\dim Y = \dim X$.
We extend Definition~\ref{d:num-Gor} to all characteristics as follows. 

\begin{defi}
A (Weil) divisor $D$ on a normal variety $X$ over an algebraically
closed field is said to be
\emph{numerically $\Q$-Cartier} if, for some regular alteration $\pi \colon Y \to X$, there is a ($\Q$-Cartier) $\Q$-divisor $D'$ on $Y$ that is $\pi$-numerically trivial (that is, it has zero
intersection with all $\pi$-exceptional curves on $Y$) and satisfies $D = \pi_* D'$.
When $D$ is numerically $\Q$-Cartier, we denote $D'$ by $\pi^*_{\rm
num}D$ and call it the \emph{numerical pull-back} of $D$ to $Y$.
\end{defi}

The fact that $D'$ in the definition is uniquely determined by the given conditions
follows by the Negativity Lemma \cite[Lemma~3.39]{KM98} applied to the Stein factorization of $\pi$. 
Since the pull-back of a Cartier divisor to a dominating regular alteration remains
relatively numerically trivial, the existence of a numerical pull-back
of $D$ to an arbitrary regular alteration follows exactly as in
\cite[Proposition 5.3]{BdFFU}, and one can replace $\p$ by an arbitrary
regular alteration in the definition. In particular, in characteristic zero the above definition
is equivalent to Definition~\ref{d:num-Gor}.

\begin{lem}
\label{lem:numercartieralterations}
Suppose that $D$ is a numerically $\Q$-Cartier $\Q$-divisor on $X$ and $\pi \colon Y \to
X$ is any regular alteration.  For any Weil divisor $E$ on $Y$, we have
\[
\O_{X}(\lfloor D \rfloor) \cdot \pi_{*}\O_{Y}(E) \subseteq \pi_{*} \O_{Y}(\lfloor E +
\pi^{*}_{\rm num}D \rfloor ).
\]
\end{lem}

\begin{proof}
We are free to assume that $X$ is
affine.  Suppose we have global sections of $\O_{X}(\lfloor D \rfloor)$ and
$\pi_{*}\O_{Y}(E)$, \textit{i.e.} we have $f
\in K(X) \subseteq K(Y)$ with $\mathrm{div}_{X}(f)+ D \geq 0$ and $g \in K(Y)$
with $\mathrm{div}_{Y}(g) + E \geq 0$.  Since $\mathrm{div}_{Y}(f) + \pi^{*}_{\rm num}D$
is $\pi$-trivial as $\mathrm{div}_{Y}(f) = \pi^{*}\mathrm{div}_{X}(f)$ and also
$\pi_{*}\left( \mathrm{div}_{Y}(f) + \pi^{*}_{\rm num}D \right) = \mathrm{div}_{X}(f) + D
\geq 0$, the Negativity Lemma once again implies that $\mathrm{div}_{Y}(f) +
\pi^{*}_{\rm num}D \geq 0$.  It follows that $\mathrm{div}_{Y}(fg) + E +
\pi^{*}_{\rm num}D \geq 0$ and so $f\cdot g$ is a global section of
$\pi_{*}\O_{Y}(\lfloor E + \pi^{*}_{\rm num}D\rfloor)$ as desired.
\end{proof}

\begin{prop}
\label{prop:easyinclusionnumgor}
  Suppose that $X$ is a normal variety over an algebraically closed
  field of positive
  characteristic.  Let $Z$ be an effective $\Q$-divisor on $X$ 
  so that $K_{X} + Z$ is numerically $\Q$-Cartier, and
  $\pi \colon Y \to X$ a regular alteration.  Then
\[
\tau(X, Z) \subseteq \Tr_{\pi} \left( \pi_{*} \O_{Y}( \lceil
  K_{Y} - \pi^{*}_{\rm num} (K_{X}+Z)\rceil\right)
\]
where $\Tr_{\pi} \colon \pi_{*} \omega_{Y} \to \omega_{X}$ is the
corresponding trace map.
\end{prop}

\newcommand{\sHom}{\mathcal{H}om}
\begin{proof}
  The statement is local, so we are free to assume $X$ is affine.  Let
  $$J = \Tr_{\pi} \left( \pi_{*} \O_{Y}( \lceil
  K_{Y} - \pi^{*}_{\rm num} (K_{X}+Z)\rceil\right) = \Tr_{\pi} \left( \pi_{*} \O_{Y}( 
  K_{Y} - \lfloor \pi^{*}_{\rm num} (K_{X}+Z)\rfloor\right).$$  It is clear that $J$
is a fractional ideal sheaf: furthermore, using that $X$ is normal, we
may verify $J \subseteq
\O_{X}$ by checking on the smooth locus of $X$ where $K_{X} +
Z$ is $\Q$-Cartier.  See \cite[Proposition 2.18]{BST} for a proof.

The conclusion follows if we can show $\phi(F^{e}_{*}J) \subseteq J$
for all $e > 0$ and all $$\phi \in \Hom_{\O_{X}}(F^{e}_{*}\O_{X}(\lceil (p^{e} - 1)
Z \rceil), \O_{X}).$$
 By use of duality for a finite morphism (see \cite{BST}), using that
 both sheaves are reflexive, we have that 
\[
\Hom_{\O_{X}}((F^{e}_{*}\O_{X}(\lceil (p^{e} - 1)
Z \rceil), \O_{X}) = F^{e}_{*}\O_{X}(\lfloor (1-p^{e})(K_{X} +
Z) \rfloor) \cdot \langle \Tr_{F^{e}} \rangle
\]
where the action on the right is premultiplication
of the
trace of Frobenius $\Tr_{F^{e}} \colon F^{e}_{*}\omega_{X} \to \omega_{X}$.  Thus, we must show
\begin{equation}
\label{eq:trace}
\Tr_{F^{e}}\left( F^{e}_{*} (\O_{X}(\lfloor (1-p^{e})(K_{X} +
Z) \rfloor) \cdot J) \right) \subseteq J.
\end{equation}
To that end, using Lemma~\ref{lem:numercartieralterations}, we have
\begin{align*}
&\O_{X}(\lfloor (1-p^{e})(K_{X} +
Z) \rfloor) \cdot J \\ \subseteq &\Tr_{\pi}\left(\pi_{*}\O_{Y}(K_{Y}
  - \lfloor \pi^{*}_{\rm num}(K_{X}+Z) \rfloor + \lfloor
  (1-p^{e})\pi^{*}_{\rm num}(K_{X} + Z) \rfloor \right) \\ \subseteq
&\Tr_{\pi}\left(\pi_{*}\O_{Y}(K_{Y}
  - p^{e}\lfloor \pi^{*}_{\rm num}(K_{X}+Z) \rfloor) \right)
\end{align*}
after verifying that for any $\Q$-divisor $D$, $-\lfloor D \rfloor +
\lfloor (1-p^{e})D \rfloor \leq -p^{e} \lfloor D \rfloor$.  Equation
\eqref{eq:trace} now follows using \cite[Proposition 2.18]{BST} and
the functoriality of the trace map:
\begin{align*}
  &\Tr_{F^{e}}\left( F^{e}_{*} \O_{X}(\lfloor (1-p^{e})(K_{X} +
Z) \rfloor) \cdot J \right) \\
\subseteq &\Tr_{F^{e}\circ \pi} \left( F^{e}_{*}\pi_{*} \O_{Y}(K_{Y}
  - p^{e}\lfloor \pi^{*}_{\rm num}(K_{X}+Z) \rfloor)\right) \\
\subseteq
&\Tr_{\pi}\left( \pi_{*}\left(\Tr_{F^{e}}(F^{e}_{*}\O_{Y}(K_{Y}
  - p^{e}\lfloor \pi^{*}_{\rm num}(K_{X}+Z) \rfloor))
\right)\right) \\
\subseteq &\Tr_{\pi}\left( \pi_{*}\O_{Y}( 
  K_{Y} - \lfloor \pi^{*}_{\rm num} (K_{X}+Z)\rfloor\right) = J.
\end{align*}
\end{proof}

  As above in characteristic zero, one may define a normal variety in
  positive characteristic to be \emph{numerically $\Q$-Gorenstein} if
  the canonical class is numerically $\Q$-Cartier.  However, given a
  model of a numerically $\Q$-Gorenstein
  variety $X$ in characteristic zero over a finitely generated
  $\Z$-algebra $A$, it is unclear whether $X_{\mu}$ remains (geometrically)
  numerically $\Q$-Gorenstein for an open subset of closed points $\mu
  \in \Spec A$.  Said another way, given the behavior
  of nefness in families (\textit{cf.} \cite{L,Les}), it is unclear if
  numerical pull-back can be preserved after reduction to positive
  characteristic.  In a sense, this is the difficulty which must be
  avoided in proving Theorem~\ref{t:main}, as otherwise one could use
  Proposition~\ref{prop:easyinclusionnumgor} in order to verify \eqref{eq:normallyeasyinclusion}.

Given the main results of \cite{BST, BSTZ}, it would seem natural to
ask the following questions.

\begin{question}
  Suppose that $X$ is a numerically $\Q$-Gorenstein variety over an algebraically closed
  field of positive
  characteristic.
  \begin{enumerate}
  \item Is it true that $\tau(X) =
    \Tr_{\pi}\left(\p_*\O_{Y}(\lceil K_{Y}-\pi^{*}_{\rm num}K_{X}\rceil)
    \right)$ for all sufficiently large regular alterations $\pi \colon Y
    \to X$?  More generally, can one find an alteration that achieves
    equality in Proposition~\ref{prop:easyinclusionnumgor}?
  \item Suppose, in addition, that $X$ is strongly $F$-regular, and let $Z$
    be a $\Q$-Cartier divisor on $X$.  Are the $F$-jumping numbers of
    the test ideals $\tau(X, \lambda Z)$ for $\lambda \in \Q_{\geq 0}$
    a discrete set of rational numbers?  In particular, is the
    $F$-pure threshold of $(X,Z)$ rational?
  \end{enumerate}
\end{question}



\begin{bibdiv}
\begin{biblist}

\bib{BST}{article}{
   author={Blickle, Manuel},
   author={Schwede, Karl},
   author={Tucker, Kevin},
   title={F-singularities via alterations},
   note={To appear in Amer. J. Math., {\tt arXiv:1107.3807}},
}

\bib{BSTZ}{article}{
   author={Blickle, Manuel},
   author={Schwede, Karl},
   author={Takagi, Shunsuke},
   author={Zhang, Wenliang},
   title={Discreteness and rationality of $F$-jumping numbers on singular
   varieties},
   journal={Math. Ann.},
   volume={347},
   date={2010},
   number={4},
   pages={917--949},
}

\bib{BdFF12}{article}{
   author={Boucksom, Sebastien},
   author={de Fernex, Tommaso},
   author={Favre, Charles},
   title={The volume of an isolated singularity},
   journal={Duke Math. J.},
   volume={161},
   date={2012},
   number={8},
   pages={1455--1520},
}

\bib{BdFFU}{article}{
   author={Boucksom, Sebastien},
   author={de Fernex, Tommaso},
   author={Favre, Charles},
   author={Urbinati, Stefano},
   title={Valuation spaces and multiplier ideals on singular varieties},
   note={To appear in the London Math. Soc. Lecture Note Series, volume in honor of Rob Lazarsfeld's 60th birthday, {\tt arXiv:1307.0227}},
}

\bib{dFH09}{article}{
   author={de Fernex, Tommaso},
   author={Hacon, Christopher D.},
   title={Singularities on normal varieties},
   journal={Compos. Math.},
   volume={145},
   date={2009},
   number={2},
   pages={393--414},
}

\bib{dJ}{article}{
   author={de Jong, A. J.},
   title={Smoothness, semi-stability and alterations},
   journal={Inst. Hautes \'Etudes Sci. Publ. Math.},
   number={83},
   date={1996},
   pages={51--93},
}

\bib{Har01}{article}{
   author={Hara, Nobuo},
   title={Geometric interpretation of tight closure and test ideals},
   journal={Trans. Amer. Math. Soc.},
   volume={353},
   date={2001},
   number={5},
   pages={1885--1906 (electronic)},
}

\bib{KM98}{book}{
   author={Koll{\'a}r, J{\'a}nos},
   author={Mori, Shigefumi},
   title={Birational geometry of algebraic varieties},
   series={Cambridge Tracts in Mathematics},
   volume={134},
   note={With the collaboration of C. H. Clemens and A. Corti;
   Translated from the 1998 Japanese original},
   publisher={Cambridge University Press},
   place={Cambridge},
   date={1998},
   pages={viii+254},
}

\bib{HH99}{article}{
    AUTHOR = {Hochster, Melvin}, 
    AUTHOR = {Huneke, Craig}, 
    TITLE = {Tight closure in equal characteristic zero}, 
    JOURNAL = {Preprint}, 
    Year = {1999}, 
}    

\bib{L}{article}{
   author={Langer, Adrian},
   title={On positivity and semistability of vector bundles in finite and mixed characteristics},
   note={To appear in J. Ramanujan Math. Soc., {\tt arXiv:1301.4450}},
}

\bib{Les}{article}{
   author={Lesieutre, John},
   title={The diminished base locus is not always closed},
   note={Submitted, {\tt arXiv:1212.37380}},
}
    
\bib{MS11}{article}{
    AUTHOR = {Musta{\c{t}}{\u{a}}, Mircea}, 
    AUTHOR = {Srinivas, Vasudevan},
    TITLE = {Ordinary varieties and the comparison between multiplier ideals and test ideals},
    JOURNAL = {Nagoya Math. J.}, 
    VOLUME = {204},
    YEAR = {2011},
    PAGES = {125--157},
}


\bib{Sch11}{article}{
   author={Schwede, Karl},
   title={Test ideals in non-$\Bbb{Q}$-Gorenstein rings},
   journal={Trans. Amer. Math. Soc.},
   volume={363},
   date={2011},
   number={11},
   pages={5925--5941},
}

\bib{Smi00}{article}{
   author={Smith, Karen E.},
   title={The multiplier ideal is a universal test ideal},
   note={Special issue in honor of Robin Hartshorne},
   journal={Comm. Algebra},
   volume={28},
   date={2000},
   number={12},
   pages={5915--5929},
}

\bib{Tak04}{article}{
   author={Takagi, Shunsuke},
   title={An interpretation of multiplier ideals via tight closure},
   journal={J. Algebraic Geom.},
   volume={13},
   date={2004},
   number={2},
   pages={393--415},
}

\end{biblist}
\end{bibdiv}
\end{document}